\documentclass[12pt, reqno]{amsart}
\usepackage{ amsmath,amsthm, amscd, amsfonts, amssymb, graphicx, color}
\usepackage[bookmarksnumbered, colorlinks, plainpages]{hyperref}
\newtheorem{thm}{Theorem}[section]
\newtheorem{cor}[thm]{Corollary}

\newtheorem{exam}[thm]{Example}
\numberwithin{equation}{section}

\begin{document}

\title{generalized Jacobson's lemma for generalized Drazin inverses}

\author{Huanyin Chen}
\author{Marjan Sheibani}
\address{
Department of Mathematics\\ Hangzhou Normal University\\ Hang -zhou, China}
\email{<huanyinchen@aliyun.com>}
\address{Women's University of Semnan (Farzanegan), Semnan, Iran}
\email{<sheibani@fgusem.ac.ir>}


\begin{abstract}
We present new generalized Jacobson's lemma for generalized Drazin inverses. This extend the main results on g-Drazin inverse of Yan, Zeng and Zhu (Linear $\&$ Multilinear Algebra, {\bf 68}(2020), 81--93).\end{abstract}

\maketitle

\section{Introduction}

Let $R$ be an associative ring with an identity. The commutant of $a\in R$ is defined by $comm(a)=\{x\in
R~|~xa=ax\}$. The double commutant of $a\in R$ is defined by $comm^2(a)=\{x\in R~|~xy=yx~\mbox{for all}~y\in comm(a)\}$. An element $a\in R$ has g-Drazin inverse in case there exists $x\in R$ such that $$x=xax, x\in comm^2(a), a-a^2x\in R^{qnil}.$$ The preceding $x$ is unique if it exists, we denote it by $a^{d}$. Here, $R^{qnil}=\{a\in R~|~1+ax\in R^{-1}~\mbox{for
every}~x\in comm(a)\}$., where $R^{-1}$ stands for the set of all invertible elements of $R$. As it is known, $a\in R$ has g-Drazin inverse if and only if there exists an idempotent $p\in comm^2(a)$ such that $a+p\in R$ is invertible and $ap\in R^{qnil}$ (see~\cite[Lemma 2.4]{K}).

For any $a,b\in R$, Jacobson's Lemma for invertibility states that $1-ab\in R^{-1}$ if and only if $1-ba\in R^{-1}$ and $(1-ba)^{-1}=1+b(1-ab)^{-1}a$ (see~\cite[Lemma 1.4]{M}). Let $a,b\in R^d$. Zhuang et al. proved the Jacobson's Lemma for g-Drazin inverse. That is, it was proved that $1-ab\in R^d$ if and only if $1-ba\in R^d$ and $$(1-ba)^d=1+b(1-ab)^da$$ (see~\cite[Theorem 2.3]{Z}). Jacobson's Lemma plays an important role in matrix and operator theory. Many papers discussed this lemma for g-Drazin inverse in the setting of matrices, operators, elements of Banach algebras or rings. Mosic generalized Jacobson's Lemma for g-Drain inverse to the case that $bdb=bac, dbd=acd$ (see~\cite[Theorem 2.5]{M}). Recently, Yan et al. extended Jacobson's Lemma to the case $dba=aca, dbd=acd$ (see~\cite[Theorem 3.3]{YZZ}). This condition was also considered for bounded linear operators in ~\cite{YF, YZY, Z1}.

The motivation of this paper is to extend the main results of Yan et al. (see~\cite{YZZ}) to a wider case. The Drazin inverse of $a\in R$, denoted by $a^D$, is the unique element $a^D$ satisfying the following three equations $$a^D=a^Daa^D, a^D\in comm(a), a^k=a^{k+1}a$$ for some $k\in {\Bbb N}$. The small integer $k$ is called the Drazin index of $a$, and is denoted by $i(a)$. Moreover, we prove the generalized Jacobson's lemma for the Drazin inverse.

Throughout the paper, all rings are associative with an identity. $R^{D}$ and $R^{d}$ denote the sets of all Drazin and g-Drazin invertible elements in $R$ respectively.
We use $R^{nil}$ to denote the set of all nilpotents of the ring $R$. ${\Bbb C}$ stands for the field of all complex numbers.

\section{generalized Jacobson's lemma}

We come now to the main result of this paper which will be the tool in our following development.

\begin{thm} Let $R$ be a ring, and let $a,b,c,d\in R$ satisfying $$\begin{array}{c}
(ac)^2=(db)(ac), (db)^2=(ac)(db);\\
b(ac)a=b(db)a, c(ac)d=c(db)d.
\end{array}$$ Then $\alpha=1-bd\in R^d$ if and only if $\beta=1-ac\in R^d$. In this case,
$$\beta^{d}=\big[1-d\alpha^{\pi}(1-\alpha (1+bd))^{-1}bac\big](1+ac)+d\alpha^dbac.$$\end{thm}
\begin{proof} Let $p=\alpha^{\pi},x=\alpha^d$. Then $1-p\alpha (1+bd)\in R^{-1}$. Let $$y=\big[1-dp(1-p\alpha (1+bd))^{-1}bac\big](1+ac)+dxbac.$$ We shall prove that
$\beta^d=y$.

Step 1. $y\beta y=y$. We see that
$$y\beta=1-(ac)^2-dp(1-p\alpha (1+bd))^{-1}bac\big[1-(ac)^2\big]+dxbac(1-ac).$$
We compute that
$${\scriptsize\begin{array}{lll}
y\beta&=&1-(ac)^2-dp(1-p\alpha (1+bd))^{-1}bac\big[1-(ac)^2\big]+dxbac(1-ac)\\
&=&1-[dbac-dxbac(1-ac)]-dp(1-p\alpha (1+bd))^{-1}[bac-bac(ac)^2\big]\\
&=&1-[dbac-dx(bac-bacac)]-dp(1-p\alpha (1+bd))^{-1}[ba-bacdba]c\\
&=&1-[dbac-dx(1-bd)bac]-dp(1-p\alpha (1+bd))^{-1}[1-(bd)^2]bac\\
&=&1-dpbac-dp(1-p\alpha (1+bd))^{-1}p\alpha (1+bd)bac\\
&=&1-dp(1-p\alpha (1+bd))^{-1}\big[(1-p\alpha (1+bd))+p\alpha (1+bd)\big]bac\\
&=&1-dp(1-p\alpha(1+bd))^{-1}bac.\\
\end{array}}$$
Since $acdbd=a(cdbd)=a(cacd)=(ac)^2d=(dbac)d=dbacd$, we have $(bacd)(bd)=(bd)(bacd)$, and so $(bacd)\alpha =\alpha (bacd).$ Hence,
$(bacd)x=x(bacd),$ and then
$$\begin{array}{ll}
&dp(1-p\alpha(1+bd))^{-1}bacdxbac(1-ac)\\
=&d(1-p\alpha(1+bd))^{-1}pxbacdbac(1-ac)\\
=&0.
\end{array}$$ Therefore we have
$${\tiny\begin{array}{ll}
&y\beta y\\
=&y-dp(1-p\alpha(1+bd))^{-1}bacy\\
=&y-dp(1-p\alpha(1+bd))^{-1}bac\big[1-dp(1-p\alpha (1+bd))^{-1}bac\big](1+ac)\\
=&y-dp(1-p\alpha(1+bd))^{-1}bac(1+ac)+dp(1-p\alpha(1+bd))^{-2}(bacd)bac(1+ac)\\
=&y-dp(1-p\alpha(1+bd))^{-1}(1+bd)bac+dp(1-p\alpha(1+bd))^{-2}(bd)^2(1+bd)bac\\
=&y-dp(1-p\alpha(1+bd))^{-2}\big[p-p\alpha(1+bd)-p(bd)^2\big](1+bd)bac\\
=&y.\\
\end{array}}$$

Step 2. $\beta-\beta y\beta\in R^{qnil}$.
Since $y=y\beta y$, we see that $(1-y\beta)^2=1-y\beta$. Hence,
$$\begin{array}{lll}
\beta-\beta y\beta&=&\beta (1-y\beta )\\
&=&\beta dp(1-p\alpha(1+bd))^{-1}(bacd)p(1-p\alpha(1+bd))^{-1}bac\\
&=&\beta d(bacd)p(1-p\alpha(1+bd))^{-2}bac\\
&=&(1-ac)d(bacd)p(1-p\alpha(1+bd))^{-2}bac\\
&=&d\alpha bacdp(1-p\alpha(1+bd))^{-2}bac.
\end{array}$$ Let $z\in comm(\beta -\beta y\beta)$. Then
$$zd\alpha bacdp(1-p\alpha(1+bd))^{-2}bac=d\alpha bacdp(1-p\alpha(1+bd))^{-2}bacz.$$
We will suffice to prove $$1+d\alpha bacdp(1-p\alpha(1+bd))^{-2}bacz\in R^{-1}.$$
Clearly, we have $$p=(bd)^2p[1-p\alpha (1+bd)]^{-1}=(bd)^4p[1-p\alpha (1+bd)]^{-2}.$$
Hence, we get
$$\begin{array}{lll}
(baczdbacd)\alpha p&=&(baczd)\alpha bacd p[1-p\alpha (1+bd)]^{-2}(bd)^4\\
&=&(baczd)\alpha bacd p[1-p\alpha (1+bd)]^{-2}bac(db)^2d\\
&=&bac\big[zd\alpha bacdp(1-p\alpha(1+bd))^{-2}bac\big]acdbd\\
&=&bac\big[d\alpha bacdp(1-p\alpha (1+bd))^{-2}bacz\big]acdbd\\
&=&bdbdbacd\big[\alpha p(1-p\alpha (1+bd))^{-2}bacz\big]acdbd\\
&=&bacdbdbd\big[\alpha p(1-p\alpha (1+bd))^{-2}bacz\big]a(cdbd)\\
&=&bdbdbdbd\big[\alpha p(1-p\alpha (1+bd))^{-2}bacz\big]a(cacd)\\
&=&bdbdbdbd\alpha p[1-p\alpha (1+bd)]^{-2}baczdbacd\\
&=&bdbdbdbd\alpha p[1-p\alpha (1+bd)]^{-2}baczdbacd\\
&=&\alpha p(baczdbacd).
\end{array}$$

Step 3. $y\in comm^2(\beta)$. Let $s\in comm(\beta)$. Then $s\beta=\beta s$, and so $s(ac)=(ac)s$.

Claim 1. $s(dxbac)=(dxbac)s.$ We easily check that
$(bacsdbd)\alpha=bacs\beta dbd\alpha=bac\beta sdbd\alpha=\alpha (bacsdbd).$ Hence $(bacsdbd)x=x(bacsdbd)$, and then
$$\begin{array}{lll}
s(dpbac)&=&sd(bd)^4p[1-p\alpha (1+bd)]^{-2}bac\\
&=&s(ac)^2dbdbdp[1-p\alpha (1+bd)]^{-2}bac\\
&=&d(bacsdbd)bdp[1-p\alpha (1+bd)]^{-2}bac\\
&=&dbdp[1-p\alpha (1+bd)]^{-2}(bacsdbd)bac\\
&=&dbdp[1-p\alpha (1+bd)]^{-2}bacs(ac)^3\\
&=&dbdp[1-p\alpha (1+bd)]^{-2}(bd)^3bacs\\
&=&d(bd)^4p[1-p\alpha (1+bd)]^{-2}bacs\\
&=&(dpbac)s.
\end{array}$$ Since $sdbac=s(ac)^2=(ac)^2s=dbacs$, we have
$$sd\alpha xbac=d\alpha xbacs,$$ and so
$$sdxbac-sdbdxbac=dxbacs-dbdxbacs.$$

On the other hand, we have
$$\begin{array}{lll}
s(dbdpbac)&=&sd(bd)^5p[1-p\alpha (1+bd)]^{-2}bac\\
&=&s(ac)^4dbdp[1-p\alpha (1+bd)]^{-2}bac\\
&=&dbdbd(bacsdbd)p[1-p\alpha (1+bd)]^{-2}bac\\
&=&dbdbdp[1-p\alpha (1+bd)]^{-2}(bacsdbd)bac\\
&=&dbdbdp[1-p\alpha (1+bd)]^{-2}bacs(ac)^3\\
&=&dbdbdp[1-p\alpha (1+bd)]^{-2}(bd)^3bacs\\
&=&dbd(bd)^4p[1-p\alpha (1+bd)]^{-2}bacs\\
&=&(dbdpbac)s.
\end{array}$$ Since $sdbdbac=s(ac)^3=(ac)^3s=dbdbacs$, we have
$$dbd\alpha xbacs=sdbd\alpha xbac.$$ Then we have $$\begin{array}{lll}
dbdbd\alpha xbacs&=&ac(dbd\alpha xbacs)\\
&=&ac(sdbd\alpha xbac)\\
&=&sacdbd\alpha xbac\\
&=&sdbdbd\alpha xbac,
\end{array}$$
and so $$dbd(1+bd)\alpha xbacs=sdbd(1+bd)\alpha xbac,$$ and then
$$dbdxbacs-dbd(bd)^2xbacs=sdbdxbac-sdbd(bd)^2xbac.$$
One easily checks that
$$\begin{array}{lll}
d(bd)^3xbacs&=&dx(bd)^3bacs\\
&=&dxb(ac)^4s\\
&=&dx(bacsdbd)bac\\
&=&d(bacsdbd)xbac\\
&=&(ac)^2sdbdxbac\\
&=&sd(bd)^3xbac.
\end{array}$$
This implies that $dbdxbacs=sdbdxbac$, and therefore $s(dxbac)=(dxbac)s.$

Claim 2. $sdp(1-p\alpha (1+bd))^{-1}bac(1+ac)=dp(1-p\alpha (1+bd))^{-1}bac(1+ac)s$.
Set $t=dp(1-p\alpha (1+bd))^{-1}bac(1+ac).$ Then we have
$$\begin{array}{lll}
st&=&sdp(1-p\alpha (1+bd))^{-1}bac(1+ac)\\
&=&sd(bd)^4p[1-p\alpha (1+bd)]^{-3}bac(1+ac)\\
&=&(ac)^3sdbdp[1-p\alpha (1+bd)]^{-3}bac(1+ac)\\
&=&dbdp[1-p\alpha (1+bd)]^{-3}bs(ac)^4(1+ac)\\
\end{array}$$
Also we have
$$\begin{array}{lll}
ts&=&dp(1-p\alpha (1+bd))^{-1}bac(1+ac)s\\
&=&dp[1-p\alpha (1+bd)]^{-3}(bd)^4bsac(1+ac)\\
&=&dbdp[1-p\alpha (1+bd)]^{-3}b(ac)^3sac(1+ac)\\
&=&dbdp[1-p\alpha (1+bd)]^{-3}bs(ac)^4(1+ac)\\
\end{array}$$
Hence, $st=ts$, and so $y\in comm^2(\beta)$. Therefore $y=\beta^d$, as required.

$\Longleftarrow$ Since $1-ac\in R^d$, it follows by Jacobson's Lemma that $1-ca\in R^d$. Applying the preceding discussion, we obtain that $1-bd\in R^d$, as desired.
\end{proof}

\begin{cor} (~\cite[Theorem 3.1]{YZZ}) Let $R$ be a ring, and let $a,b,c,d\in R$ satisfying $$acd=dbd, dba=aca.$$ Then $1-bd\in R^{d}$ if and only if
$1-ac\in R^d$. In this case, $$(1-bd)^{d}=\big[1-dp(1-p\alpha (1+bd))^{-1}bac\big](1+ac)+d(1-ac)^dbac.$$\end{cor}
\begin{proof} By hypothesis, we easily check that
$$\begin{array}{c}
(ac)^2=(aca)c=(dba)c=(db)(ac),\\
(db)^2=(dbd)b=(acd)b=(ac)(db);\\
b(ac)a=b(aca)=b(dba)=b(db)a,\\
c(ac)d=c(acd)=c(dbd)=c(db)d.
\end{array}$$ There the result follows by Theorem 2.1.\end{proof}

We now generalize ~\cite[Corollary 3.5]{YZZ} as follows.

\begin{cor} Let $R$ be a ring, and let $a,b,c\in R$ satisfying $$\begin{array}{c}
(aba)b=(aca)b, b(aba)=b(aca),\\
(aba)c=(aca)c, c(aba)=c(aca).
\end{array}$$ Then $1-ba\in R^{d}$ if and only if
$1-ac\in R^d$. In this case, $$(1-ba)^{d}=\big[1-ap(1-p\alpha (1+ba))^{-1}bac\big](1+ac)+a(1-ac)^dbac.$$\end{cor}
\begin{proof} By hypothesis, we verify that
$$\begin{array}{c}
(ac)^2=(aca)c=(aba)c=(ab)(ac),\\
(ab)^2=(aba)b=(aca)b=(ac)(db);\\
b(ac)a=b(aca)=b(aba)=b(ab)a,\\
c(ac)a=c(aca)=c(aba)=c(ab)a.
\end{array}$$ This completes the proof by Theorem 2.1.\end{proof}

It is convenient at this stage to derive the following.

\begin{thm} Let $R$ be a ring, let $n\in {\Bbb N}$, and let $a,b,c,d\in R$ satisfying $$\begin{array}{c}
(ac)^2=(db)(ac), (db)^2=(ac)(db);\\
b(ac)a=b(db)a, c(ac)d=c(db)d.
\end{array}$$ Then $(1-bd)^n\in R^d$ if and only if $(1-ac)^n\in R^d$.\end{thm}
\begin{proof}  $\Longrightarrow$ Let $\alpha=(1-ac)^n$. Then $$\begin{array}{lll}
\alpha&=&\sum\limits_{i=0}^{n}(-1)^{i}\left(
\begin{array}{c}
n\\
i
\end{array}
\right)(ac)^{i}\\
&=&1-a\sum\limits_{i=1}^{n}(-1)^i\left(
\begin{array}{c}
n\\
i
\end{array}
\right)c(ac)^{i-1}\\
&=&1-ac',
\end{array}$$ where $c'=\sum\limits_{i=1}^{n}(-1)^i\left(
\begin{array}{c}
n\\
i
\end{array}
\right)c(ac)^{i-1}$. Let $\beta=(1-ba)^n$. Then
$$\begin{array}{lll}
\beta&=&\sum\limits_{i=0}^{n}(-1)^{i}\left(
\begin{array}{c}
n\\
i
\end{array}
\right)(bd)^{i}\\
&=&1-\big(\sum\limits_{i=1}^{n}(-1)^i\left(
\begin{array}{c}
n\\
i
\end{array}
\right)(bd)^{i-1}b\big)d\\
&=&1-b'd,
\end{array}$$ where $b'=\sum\limits_{i=1}^{n}(-1)^i\left(
\begin{array}{c}
n\\
i
\end{array}
\right)(bd)^{i-1}b.$
We directly compute that
$$\begin{array}{ll}
&(ac')^2\\
=&\big[a\sum\limits_{i=1}^{n}(-1)^i\left(
\begin{array}{c}
n\\
i
\end{array}
\right)c(ac)^{i-1}\big]^2=\big[\sum\limits_{i=1}^{n}(-1)^i\left(
\begin{array}{c}
n\\
i
\end{array}
\right)(ac)^{i}\big]^2\\
=&\big[\sum\limits_{i=1}^{n}(-1)^i\left(
\begin{array}{c}
n\\
i
\end{array}
\right)(ac)^{i}\big]ac\big[\sum\limits_{i=1}^{n}(-1)^i\left(
\begin{array}{c}
n\\
i
\end{array}
\right)(ac)^{i-1}\big];\\
&(db')(ac')\\
=&\big[\sum\limits_{i=1}^{n}(-1)^i\left(
\begin{array}{c}
n\\
i
\end{array}
\right)d(bd)^{i-1}b\big]\big[\sum\limits_{i=1}^{n}(-1)^i\left(
\begin{array}{c}
n\\
i
\end{array}
\right)(ac)^{i}\big]\\
=&\big[\sum\limits_{i=1}^{n}(-1)^i\left(
\begin{array}{c}
n\\
i
\end{array}
\right)(db)^{i}\big]ac\big[\sum\limits_{i=1}^{n}(-1)^i\left(
\begin{array}{c}
n\\
i
\end{array}
\right)(ac)^{i-1}\big].
\end{array}$$ Since $(ac)^i(ac)=(db)^i(ac)$ for any $i\in {\Bbb N}$, we have $(ac')^2=(db')(ac').$ Moreover, we check that
$$\begin{array}{ll}
&(db')^2\\
=&\big[d\sum\limits_{i=1}^{n}(-1)^i\left(
\begin{array}{c}
n\\
i
\end{array}
\right)(bd)^{i-1}b\big]^2=\big[\sum\limits_{i=1}^{n}(-1)^i\left(
\begin{array}{c}
n\\
i
\end{array}
\right)(db)^{i}\big]^2\\
=&\big[\sum\limits_{i=1}^{n}(-1)^i\left(
\begin{array}{c}
n\\
i
\end{array}
\right)(db)^{i}\big]db\big[\sum\limits_{i=1}^{n}(-1)^i\left(
\begin{array}{c}
n\\
i
\end{array}
\right)(db)^{i-1}\big];\\
&(ac')(db')\\
=&\big[a\sum\limits_{i=1}^{n}(-1)^i\left(
\begin{array}{c}
n\\
i
\end{array}
\right)c(ac)^{i-1}\big]\big[d\sum\limits_{i=1}^{n}(-1)^i\left(
\begin{array}{c}
n\\
i
\end{array}
\right)b(db)^{i-1}\big]\\
=&\big[\sum\limits_{i=1}^{n}(-1)^i\left(
\begin{array}{c}
n\\
i
\end{array}
\right)(ac)^{i}\big]db\big[\sum\limits_{i=1}^{n}(-1)^i\left(
\begin{array}{c}
n\\
i
\end{array}
\right)(db)^{i-1}\big].
\end{array}$$ Since $(ac)^i(db)=(db)^i(db)$ for any $i\in {\Bbb N}$, we have $(db')^2=(ac')(db')$. Furthermore, we verify that $$\begin{array}{c}
b'(ac')a=\big[\sum\limits_{i=1}^{n}(-1)^i\left(
\begin{array}{c}
n\\
i
\end{array}
\right)(bd)^{i-1}b\big]\big[a\sum\limits_{i=1}^{n}(-1)^i\left(
\begin{array}{c}
n\\
i
\end{array}
\right)c(ac)^{i-1}a\big]\\
=\big[\sum\limits_{i=1}^{n}(-1)^i\left(
\begin{array}{c}
n\\
i
\end{array}
\right)(bd)^{i-1}\big]\big[\sum\limits_{i=1}^{n}(-1)^i\left(
\begin{array}{c}
n\\
i
\end{array}
\right)b(ac)^{i}a\big];\\
b'(db')a=\big[\sum\limits_{i=1}^{n}(-1)^i\left(
\begin{array}{c}
n\\
i
\end{array}
\right)(bd)^{i-1}b\big]\big[d\sum\limits_{i=1}^{n}(-1)^i\left(
\begin{array}{c}
n\\
i
\end{array}
\right)(bd)^{i-1}\big]a\\
=\big[\sum\limits_{i=1}^{n}(-1)^i\left(
\begin{array}{c}
n\\
i
\end{array}
\right)(bd)^{i-1}\big]\big[\sum\limits_{i=1}^{n}(-1)^i\left(
\begin{array}{c}
n\\
i
\end{array}
\right)b(db)^{i}a\big].
 \end{array}$$ Since $b(ac)a=b(db)a$ we see that $b(ac)^2a=b(dbac)a=bdb(ac)a=bdb(db)a=b(dbdb)a=b(db)^2a.$ By induction, we have
$b(ac)^{i}a=b(db)^{i}a$ for any $n\in {\Bbb N}$. Therefore $b'(ac')a=b'(db')a.$ Also we have
$$\begin{array}{c}
c'(ac')d=\big[\sum\limits_{i=1}^{n}(-1)^ic(ac)^{i-1}\big]\big[a\sum\limits_{i=1}^{n}(-1)^ic(ac)^{i-1}d\big]\\
=\big[\sum\limits_{i=1}^{n}(-1)^i(ca)^{i-1}\big]\big[\sum\limits_{i=1}^{n}(-1)^ic(ac)^{i}d\big];\\
c'(db')d=\big[\sum\limits_{i=1}^{n}(-1)^ic(ac)^{i-1}\big]\big[d\sum\limits_{i=1}^{n}(-1)^i(bd)^{i-1}b\big]d\\
=\big[\sum\limits_{i=1}^{n}(-1)^i(ca)^{i-1}\big]\big[\sum\limits_{i=1}^{n}(-1)^ic(db)^{i}d\big].
\end{array}$$ Since $c(ac)d=c(db)d$, by induction, we get $c(ac)^id=c(db)^id$ for any $n\in {\Bbb N}$.
This implies that $c'(ac')d=c'(db')d$. In light of Theorem 2.1, $1-db'\in R^d$ if and only if $1-ac'\in R^d$, as desired.

$\Longleftarrow$ Since $(1-ac)^n\in R^d$, applying the preceding discussion, $(1-ca)^n\in R^d$ and then
$(1-db)^n\in R^d$. Therefore $(1-bd)^n\in R^d$, the result follows.\end{proof}

\begin{cor} Let $R$ be a ring, let $n\in {\Bbb N}$, and let $a,b,c\in R$ satisfying $$\begin{array}{c}
(aba)b=(aca)b, b(aba)=b(aca),\\
(aba)c=(aca)c, c(aba)=c(aca).
\end{array}$$ Then $(1-ba)^n\in R^{d}$ if and only if
$(1-ac)^n\in R^d$.\end{cor}\begin{proof} This is obvious by Theorem 2.4.\end{proof}

\section{Drazin inverse}

As it is known, $a\in R^D$ if and only if there exists $x\in R$ such that $x=xax, x\in comm^2(a), a-a^2x\in R^{nil}$, and so $a^D=a^d$. For the generalized Jacobson's Lemma for Drazin inverse, we have

\begin{thm} Let $R$ be a ring, and let $a,b,c,d\in R$ satisfying $$\begin{array}{c}
(ac)^2=(db)(ac), (db)^2=(ac)(db);\\
b(ac)a=b(db)a, c(ac)d=c(db)d.
\end{array}$$ Then $1-bd\in R^D$ if and only if $1-ac\in R^D$. In this case,
$$\begin{array}{lll}
(1-ac)^{D}&=&\big[1-d(1-bd)^{\pi}(1-\alpha (1+bd))^{-1}bac\big](1+ac)\\
&+&d(1-bd)^Dbac,\\
i(1-bd)&\leq &i(1-ac)+1.
\end{array}$$
\end{thm}
\begin{proof} Set $\alpha=1-bd$ and $\beta=1-ac$. Let $p=\alpha^{\pi},x=\alpha^D$. In view of Theorem 2.1, $\beta\in R^d$ and $$\beta^d=\big[1-dp(1-p\alpha (1+bd))^{-1}bac\big](1+ac)+dxbac.$$ We shall prove that $\beta^D=\beta^d$.

We will suffice to check $\beta-\beta \beta^d\beta\in R^{nil}$.
As in the proof of Theorem 2.1, we have
$$\begin{array}{lll}
\beta-\beta \beta^d\beta&=&\beta (1-\beta^d\beta )\\
&=&d\alpha bacdp(1-p\alpha(1+bd))^{-2}bac.
\end{array}$$ In light of~\cite[Lemma 3.1]{Mi}, we will suffice to prove $$bacd\alpha bacdp(1-p\alpha(1+bd))^{-2}\in R^{nil}.$$
Similarly to the discussion in Theorem 2.1, we see that $bacd\in comm(\alpha)$, and so $bacd, \alpha p$ and $(1-p\alpha(1+bd))^{-2}$ commute one another.
Set $n=i(\alpha)$. Then $$\begin{array}{ll}
&\big[bacd\alpha bacdp(1-p\alpha(1+bd))^{-2}\big]^n\\
=&(bacd)^2(1-p\alpha(1+bd))^{-2n}(\alpha-\alpha^2\alpha^d)^n\\
=&0;
\end{array}$$ hence,
$$\begin{array}{l}
(\beta-\beta \beta^d\beta)^{n+1}\\
=d\alpha bacdp(1-p\alpha(1+bd))^{-2}\big[bacd\alpha bacdp(1-p\alpha(1+bd))^{-2}\big]^nbac\\
=0.
\end{array}$$ Thus we have
$\beta-\beta \beta^d\beta\in R^{nil}$, and so $\beta^D=\beta^d$. Moreover, we have $i(\beta)\leq i(\alpha)+1$, as desired.\end{proof}

As an immediate consequence of Theorem 3.1, we now derive

\begin{cor} Let $R$ be a ring, and let $a,b,c,d\in R$ satisfying $$acd=dbd, dba=aca.$$ Then $1-bd\in R^{D}$ if and only if
$1-ac\in R^D$. In this case, $$\begin{array}{lll}
(1-ac)^{D}&=&\big[1-d\alpha^{\pi}(1-\alpha (1+bd))^{-1}bac\big](1+ac)\\
&+&d(1-bd)^Dbac,\\
i(1-bd)&\leq&i(1-ac)+1.
\end{array}$$\end{cor}

\begin{cor} Let $R$ be a ring, and let $a,b,c\in R$ satisfying $$\begin{array}{c}
(aba)b=(aca)b, b(aba)=b(aca),\\
(aba)c=(aca)c, c(aba)=c(aca).
\end{array}$$ Then $1-ba\in R^D$ if and only if $1-ac\in R^D$. In this case,
$$\begin{array}{lll}
(1-ac)^{D}&=&\big[1-a(1-ba)^{\pi}(1-(1-ba)(1+ba))^{-1}bac\big](1+ac)\\
&+&d(1-ba)^Dbac,\\
i(1-ba)&\leq &i(1-ac)+1.
\end{array}$$\end{cor}

The group of $a\in R$ is the unique element $a^{\#}\in R$ which satisfies $a=aa^{\#}a, a^{\#}=a^{\#}aa^{\#}.$ We denote the set of all group invertible elements of $R$ by $R^{\#}$. As it is well known, $a\in R^{\#}$ if and only if $a\in R^D$ and $i(a)=1$. We are now ready to prove:

\begin{thm} Let $R$ be a ring, and let $a,b,c,d\in R$ satisfying $$\begin{array}{c}
(ac)^2=(db)(ac), (db)^2=(ac)(db);\\
b(ac)a=b(db)a, c(ac)d=c(db)d.
\end{array}$$ Then $1-bd$ has group inverse if and only if $1-ac$ has group inverse. In this case, $$\begin{array}{ll}
&(1-ac)^{\#}\\
=&\big[1-d\alpha^{\pi}(1-\alpha (1+bd))^{-1}bac\big](1+ac)+d(1-bd)^{\#}bac.
\end{array}$$\end{thm}
\begin{proof} Since $1-bd\in R^{\#}$, we have $1-bd\in R^D$. In light of Theorem 3.1, $1-ac\in R^D$.
Let $\alpha=1-bd$ and $\beta=1-ac$. Let $p=1-\alpha\alpha^D$. Since $\alpha\in R^{\#}$, we have $\alpha p=\alpha-\alpha^2\alpha^D=0$. As in the proof of Theorem 3.1,
we have $$\begin{array}{ll}
&\beta-\beta \beta^D\beta\\
=&d\alpha (bacd)p(1-p\alpha(1+bd))^{-2}bac\\
=&d(bacd)\alpha p(1-p\alpha(1+bd))^{-2}bac\\
=&0.
\end{array}$$ Obviously, $\beta^D\in comm(\beta)$ and $\beta^D=\beta^D\beta\beta^D$.
Therefore $$\begin{array}{l}
\beta^{\#}=\beta^D\\
=\big[1-d\alpha^{\pi}(1-\alpha (1+bd))^{-1}bac\big](1+ac)+d(1-bd)^{\#}bac.
\end{array}$$ This completes the proof.\end{proof}

\begin{cor} Let $R$ be a ring, and let $a,b,c\in R$ satisfying $$\begin{array}{c}
(aba)b=(aca)b, b(aba)=b(aca),\\
(aba)c=(aca)c, c(aba)=c(aca).
\end{array}$$  Then $1-ba$ has group inverse if and only if $1-ac$ has group inverse. In this case, $$\begin{array}{ll}
&(1-ac)^{\#}\\
=&\big[1-a(1-ba)^{\pi}(1-\alpha (1+ba))^{-1}bac\big](1+ac)+a(1-ba)^{\#}bac,
\end{array}$$
\end{cor}
\begin{proof} This is obvious by Theorem 3.4.\end{proof}

Corollary 3.5 is a nontrivial generalization of ~\cite[Corollary 2.4]{M} as the following example follows.

\begin{exam}\end{exam} Let $R=M_2({\Bbb C})$. Choose $$a=
\left(
\begin{array}{cc}
1&1\\
1&0
\end{array}
\right), b=\left(
\begin{array}{cc}
1&-1\\
0&0
\end{array}
\right), c=\left(
\begin{array}{cc}
0&0\\
0&0
\end{array}
\right)
\in R$$ Then we see that
$$\begin{array}{c}
(aba)b=(aca)b, b(aba)=b(aca),\\
(aba)c=(aca)c, c(aba)=c(aca).
\end{array}$$ But $aba=\left(
\begin{array}{cc}
0&1\\
0&1
\end{array}
\right)\neq 0=aca.$ In this case, $$(1-ac)^{\#}=\left(
\begin{array}{cc}
1&0\\
0&1
\end{array}
\right), (1-ba)^{\#}=\left(
\begin{array}{cc}
1&-1\\
0&1
\end{array}
\right).$$

\end{document}